\documentclass[12pt]{article}
\usepackage[top=3.3cm,bottom=3.2cm,left=3.2cm,right=3.2cm]{geometry}
\usepackage[utf8]{inputenc}
\usepackage{amssymb,amsmath,amsfonts,eucal,mathrsfs,amsthm,xcolor} 
\usepackage{enumitem}
\usepackage{amsrefs}
\usepackage{stmaryrd}
\usepackage{hyperref}
\hypersetup{
    colorlinks=true,
    filecolor=magenta,      
    linkcolor={blue!50!black},
    citecolor={red!50!black},
    urlcolor={blue!80!black}
}

\newtheorem{theorem}{Theorem}
\newtheorem{proposition}[theorem]{Proposition}
\newtheorem{lemma}[theorem]{Lemma}
\newtheorem{sublemma}[theorem]{Sublemma}
\newtheorem{corollary}[theorem]{Corollary}

\newtheorem{remark}[theorem]{Remark}

\newtheorem{example}[theorem]{Example}

\theoremstyle{definition}

\newcommand{\CU}{\mathcal{U}}
\newcommand{\Vol}{\mathrm{Vol}}
\newcommand{\R}{\mathbb{R}}
\newcommand{\Fi}{\mathbb{F}}
\newcommand{\Q}{\mathbb{Q}}
\newcommand{\Sf}{\mathbb{S}}

\newcommand{\spa}{\mbox{span}}

\newcommand{\grad}{\mbox{grad\,}}

\newcommand{\Hom}{\mbox{Hom}}

\newcommand{\tr}{\mathrm{tr}}

\newcommand{\E}{{\cal E}}
\def\span{{\rm{span}}}
\def\bea{\begin{eqnarray*} }
\def\eea{\end{eqnarray*} }

\def\beq{\begin{equation}}
\def\Z{\mathord{\mathbb Z}}
\def\N{\mathord{\mathbb N}}

\def\<{{\langle}}
\def\>{{\rangle}}

\def\a{\alpha}

\def\be{\begin{equation} }
\def\ee{\end{equation} }

\def\proof{\noindent{\it Proof:  }}
\def\qed{\ifhmode\unskip\nobreak\fi\ifmmode\ifinner
\else\hskip5 pt \fi\fi\hbox{\hskip5 pt \vrule width4 pt
height6 pt  depth1.5 pt \hskip 1pt }}

\begin{document}

\title{Partial Scalar Curvatures and Topological Obstructions for Submanifolds}
\author{C.-R. Onti, K. Polymerakis and Th. Vlachos}
\date{}
\maketitle

\renewcommand{\thefootnote}{\fnsymbol{footnote}} 
\renewcommand{\thefootnote}{\arabic{footnote}} 

\renewcommand{\thefootnote}{\fnsymbol{footnote}} 
\renewcommand{\thefootnote}{\arabic{footnote}} 

\begin{abstract}
We investigate specific intrinsic curvatures $\rho_k$, 
$1\leq k\leq n$, that interpolate between the 
minimum Ricci curvature $\rho_1$ and the normalized 
scalar curvature $\rho_n=\rho$ of $n$-dimensional 
Riemannian manifolds. For $n$-dimensional submanifolds in space 
forms, these curvatures satisfy an
inequality  involving the mean curvature 
$H$ and the normal scalar curvature 
$\rho^\perp$, which reduces to the well-known DDVV inequality when 
$k=n$. We derive topological obstructions for compact 
$n$-dimensional submanifolds based on universal 
lower bounds of the 
$L^{n/2}$-norms of certain functions involving 
$\rho_k,H$ and $\rho^\perp$. 
These obstructions are expressed in terms of the Betti numbers.
Our main result applies for any 
$1\leq k \leq n-1$, but it generally fails for 
$k=n$, where the involved norm vanishes precisely 
for Wintgen ideal submanifolds. We demonstrate this 
by providing a method of constructing new compact 
3-dimensional minimal Wintgen ideal submanifolds in 
even-dimensional spheres. Specifically, we prove 
that such submanifolds exist in $\Sf^6$ 
with arbitrarily large first Betti number.
\end{abstract}

\section{Introduction}

A fundamental problem in differential geometry is to investigate 
the intricate relationship between the geometry and topology 
of Riemannian manifolds. From the perspective of submanifold 
theory, there is a profound interplay between the intrinsic and 
extrinsic geometries of a submanifold, significantly influencing 
its topological properties. In particular, the intrinsic geometry, 
governed by metrics and curvature tensors intrinsic to the 
submanifold itself, interacts with the extrinsic geometry, 
characterized by how the submanifold is embedded within a 
higher-dimensional ambient space. This interaction can lead 
to profound insights into the topological structure of the 
submanifold.

One of the key questions in this area involves understanding 
how intrinsic curvature invariants, such as the Ricci curvature 
and scalar curvature, relate to extrinsic ones, 
including the mean curvature and normal curvature. 
These relationships often manifest through inequalities that 
bind intrinsic quantities to extrinsic ones, providing 
constraints that influence the topological classification 
of the submanifolds. 

In the case of surfaces $f\colon M^2\to\Q_c^{2+m}$ 
isometrically immersed into the complete simply connected 
space form $\Q_c^{2+m}$ of constant curvature $c$, such 
a pointwise relation is Wintgen's inequality
$$
K\leq c+ H^2 -|K_N|. 
$$
This inequality, first proved by Wintgen \cite{Wi} for surfaces in 
$\R^4$, provides an extrinsic upper bound for the Gaussian curvature 
$K$ of the surface in terms of the length 
$H$ of the mean curvature vector field and the normal curvature 
$K_N$ of $f$, which is related to the area of the curvature ellipse. 
Furthermore, Wintgen demonstrated that equality holds 
precisely at the points where the ellipse of curvature is a circle. 
Guadalupe and Rodriguez \cite{GR} extended Wintgen's inequality 
in the aforementioned form and straightforwardly derived an 
inequality for compact surfaces. This inequality establishes a 
relationship between the integral of $c+ H^2 -|K_N|$ and 
the Euler-Poincaré characteristic of the surface.

Wintgen's inequality has been generalized to any 
isometric immersion $f\colon M^n\to\Q_c^{n+m}$ with arbitrary 
dimension into space forms, resulting in what is known as the 
DDVV (cf. \cite{DDVV, GT, Lu}) inequality 
$$
\rho\leq c+H^2-\rho^\perp.
$$
This pointwise inequality provides an upper bound
for the normalized scalar curvature $\rho$  of  $M^n$ in 
terms of the length $H$ of the mean curvature vector field, 
and of the normal scalar curvature of $f$ given by 
$\rho^\perp=\|R^\perp\|/n(n-1)$, where $R^\perp$ is the 
normal curvature tensor of $f$. Submanifolds for which 
the DDVV inequality holds as equality at every point are 
termed Wintgen ideal submanifolds. These submanifolds 
often exhibit highly symmetric structures, and their 
classification remains an area of active research, 
especially in higher dimensions (see \cite{XLMW} and references therein). 
In the case $n=2$, these precisely correspond to superconformal surfaces, 
which are surfaces with curvature ellipses becoming circles 
at every point. However, in dimensions $n\geq3$, the 
classification of Wintgen ideal submanifolds remains an 
intriguing yet challenging problem.

In this paper, we are interested in a type of partial scalar 
curvatures essentially introduced by Wolfson in \cite{Wo}. 
More precisely, let $M^n$ be a Riemannian 
manifold and denote by $\lambda_1\leq\dots\leq\lambda_n$
the eigenvalues of the normalized Ricci tensor at each point.
The {\emph{$k$-th scalar curvature}} of $M^n$ is defined 
pointwise as
$$
\rho_k=\frac{1}{k}\sum_{i=1}^k\lambda_i,\;\;1\leq k\leq n.
$$
At any point, the $k$-th scalar curvature interpolates between the 
minimum $\rho_1$ of the normalized Ricci curvature for $k=1$, 
and the normalized scalar curvature $\rho=\rho_n$.
In fact, $\rho_k$ is the minimum of the average of the sum of 
Ricci curvatures in $k$ orthonormal vectors, 
and it is a continuous function on $M^n$ for every $k$.

Using the DDVV inequality and its equality case (cf. \cite{GT, Lu}), 
one can show that for every $n$-dimensional submanifold $f\colon M^n\to \Q^{n+m}_c$, $n\geq 3$,
the following inequality holds
$$
\rho_k\leq c+H^2-\rho^\perp,\;\;1\leq k\leq n-1,
$$
at every point of $M^n$, where the equality holds precisely at the umbilical points.

The aim of the present paper is to derive topological obstructions for compact 
$n$-dimensional submanifolds based on universal 
lower bounds of the $L^{n/2}$-norms of certain functions involving 
$\rho_k,H$ and $\rho^\perp$. 
These obstructions are expressed in terms of the Betti numbers. 
For compact submanifolds into space forms $\Q^{n+m}_c$ with $c\geq0$, 
we prove the following theorem that can be regarded as a generalization of the 
result of Guadalupe-Rodriguez \cite{GR} in the context of higher-dimensional 
submanifolds. Throughout this paper, we assume all considered manifolds 
to be connected without boundary and oriented, with $i$-th 
Betti number $b_{i}(M^n;\mathbb{F})$ over an arbitrary coefficient 
field $\mathbb{F}$.

\begin{theorem}\label{main1}
Given integers $n\geq 3$ and $m\geq 1$, there exists for every 
$k\in\{1,\dots,n-1\}$ a positive constant $\varepsilon_k(n,m)$, 
depending only on $n$ and $m$, such that if $M^n$ is a compact 
Riemannian manifold that admits an isometric immersion into 
$\Q_c^{n+m}$ with $c\geq 0$, 
then
$$
\int_{M^n}(c+H^2 -\rho^\perp-\rho_k)^{n/2}dM
\geq\varepsilon_k(n,m)\sum_{i=1}^{n-1} b_{i}(M^n;\mathbb{F})
$$
for any coefficient field $\mathbb{F}$. Moreover, if 
$$\int_{M^n}(c+H^2 -\rho^\perp -\rho_k)^{n/2}dM<2\varepsilon_k(n,m)$$
for some $k\in\{1,\dots,n-1\}$, then $M^n$ is either homeomorphic to 
$\Sf^n$, or it is an Eells-Kuiper manifold.
In particular, $M^n$ is homeomorphic to $\Sf^n$ if 
$$
\int_{M^n}(c+H^2-\rho^\perp-\rho_k)^{n/2}dM<\varepsilon_k(n,m).
$$
\end{theorem}

We recall that an Eells-Kuiper manifold \cite{EK} of dimension $n$ is a 
compactification of $\R^n$ by a sphere of dimension $n/2$, where 
$n=2,4,8$, or $16$. For $n\geq 4$, it is simply connected and has the integral 
cohomology structure of the complex ($n=4$), quaternionic ($n=8$), or the 
Cayley ($n=16$) projective plane. 

We point out that the assumption $n\geq 3$ in the above theorem is 
essential, and the result does not hold in the case $n=2$. 
In fact, the involved integral vanishes for every compact 
superconformal surface, and there exist an abundance of such 
surfaces with arbitrary genus in any codimension $m\geq 2$. 

Chen and Wei in \cite{CW} obtained a geometric rigidity result 
for compact submanifolds $f\colon M^n\to\Q^{n+m}_c,n\geq 4$, 
with parallel mean curvature vector field, in terms of the lowest 
eigenvalue of the Ricci tensor. More precisely, for $c\geq0$, they 
proved that such a submanifold is totally umbilical if  
the $L^{n/2}$-norm of the nonnegative part of the function 
$\lambda-(n-1)\rho_1$ is sufficiently pinched, 
where $\lambda$ is a constant satisfying 
$(n-2)(c+H^2)<\lambda\leq (n-1)(c+H^2)$. 
The following immediate consequence of Theorem \ref{main1} 
is the analogous topological rigidity result for submanifolds of 
dimension $n\geq 3$, without imposing any assumption on 
the mean curvature.

\begin{corollary}\label{Wei}
Let  $f\colon M^n\to\Q_c^{n+m},n\geq 3,c\geq 0$, be an isometric 
immersion of a compact Riemannian manifold. If  
$$\int_{M^n}(c+H^2-\rho^\perp-\rho_1)^{n/2}dM<\varepsilon_1(n,m),$$
then $M^n$ is homeomorphic to $\Sf^n$.
\end{corollary}

The following result is a consequence 
of Theorem \ref{main1} that applies to Einstein submanifolds.

\begin{corollary}\label{Einstein}
Let  $f\colon M^n\to\Q_c^{n+m},n\geq 3, c\geq 0$, be an isometric 
immersion of a compact Einstein manifold with Ricci 
curvature $\rho$. Then
$$
\int_{M^n}(c+H^2-\rho^\perp-\rho)^{n/2}dM
\geq\varepsilon_1(n,m)\sum_{i=1}^{n-1}b_{i}(M^n;\mathbb{F})
$$
for any coefficient field $\Fi$. Moreover, if 
$$\int_{M^n}(c+H^2-\rho^\perp-\rho)^{n/2}dM<2\varepsilon_1(n,m),$$
then $M^n$ is either homeomorphic to $\Sf^n$, or it is an 
Eells-Kuiper manifold. In particular, $M^n$ is homeomorphic to $\Sf^n$ if 
$$\int_{M^n}(c+H^2-\rho^\perp-\rho)^{n/2}dM<\varepsilon_1(n,m).$$
\end{corollary}

The above result demonstrates that Theorem \ref{main1} is also valid for $k=n$ 
in the case of any Einstein submanifold, as $\rho_k=\rho$ for any 
$1\leq k\leq n$  in this scenario. However, Theorem \ref{main1} doesn't hold 
for $k=n$ in general, in which case the involved integral vanishes for 
Wintgen ideal submanifolds. We note that apart from superconformal 
surfaces and totally umbilical submanifolds, 
there are only a few known compact Wintgen ideal submanifolds
(cf. \cite{DDVV, XLMW}) and all of them have dimension $n=3$.
In Section 5, we provide counterexamples to Theorem 1 for 
$k=n=3$. In fact, we present a method to generate new compact 
Wintgen ideal submanifolds with positive first Betti number, 
which is an interesting result in its own right.
These Wintgen ideal submanifolds  
are unit bundles of plane subbundles of the normal bundle of
appropriate compact minimal surfaces in even-dimensional 
spheres. This class of minimal surfaces encompasses 
pseudoholomorphic curves in the nearly K\"ahler sphere 
$\Sf^6$. This justifies the inclusion of the additional term in 
the following result when considering the case $k=n$.

\begin{theorem}\label{main2}
Given integers $n\geq 2$ and $m\geq 1$, there exists for every 
$\lambda\in [0,1)$ a positive constant $\varepsilon_\lambda(n,m)$, 
depending only on $n$ and $m$, 
such that if $M^n$ is a compact 
Riemannian manifold that admits an isometric immersion 
into $\Q_c^{n+m}$, then
$$
\int_{M^n}(c+H^2 -\lambda\rho^\perp-\rho)^{n/2}dM
\geq\varepsilon_\lambda(n,m)\sum_{i=1}^{n-1} b_{i}(M^n;\mathbb{F})
$$
for any coefficient field $\Fi$. Moreover, if 
$$
\int_{M^n}(c+H^2-\lambda\rho^\perp-\rho)^{n/2}dM<2\varepsilon_\lambda(n,m)
$$ 
for some $\lambda\in [0,1)$, then $M^n$ is either homeomorphic 
to $\Sf^n$, or it is an Eells-Kuiper manifold. In particular, $M^n$ is homeomorphic to $\Sf^n$ if 
$$
\int_{M^n}(c+H^2-\lambda\rho^\perp-\rho)^{n/2}dM<\varepsilon_\lambda(n,m).
$$ 
\end{theorem}

It is worth noting that the integral in Theorem \ref{main2} is conformally 
invariant, rendering Theorem \ref{main2} applicable to submanifolds 
in any space form. Moreover, the above theorem may be regarded 
as an enhancement of a topological result by Shiohama and Xu 
\cite{SX2} for submanifolds in any space form, where the case 
$\lambda=0$ is considered.

The paper is organized as follows: In Section 2, we prove 
an auxiliary lemma that is essential for the proofs of our 
results and may have potential applications in other contexts. 
In Section 3, we derive algebraic inequalities 
regarding symmetric bilinear forms, which are crucial 
to our proofs. Additionally, we provide  examples that 
suggest Theorem \ref{main1} may not 
hold in general for $k=n>3$. In Section 4, we give the 
detailed proofs of Theorems \ref{main1} and \ref{main2}. 
The final Section 5 is devoted to the construction of the 
aforementioned new compact 3-dimensional Wintgen ideal 
submanifolds with arbitrarily large first Betti number, providing 
counterexamples to Theorem \ref{main1} for $k=n=3$.

\section{An auxiliary lemma}

The aim of this section is to prove an algebraic lemma essential 
to our results. In the following, let $V$ and $W$ be  
finite-dimensional real vector spaces with dimensions $n\geq 2$ 
and $m\geq 1$ respectively. Both spaces are equipped with 
positive definite inner products, which, for convenience, we 
denote by the same symbol $\langle\cdot,\cdot\rangle$. We 
denote the space of all symmetric bilinear forms from $V$ to $W$ 
by $\mathrm{Sym}(V\times V,W)$, and the space of all 
self-adjoint endomorphisms of $V$ by $\mathrm{End}(V)$. 
We recall that the {\it index} of an element $A\in \mathrm{End}(V)$
is the number of the negative eigenvalues of $A$, and it is 
denoted by $\mathrm{Index}\, A$.
Given $\beta\in\mathrm{Sym}(V\times V,W)$, we associate 
to each $\xi\in W$ the formal shape operator 
$A_\xi(\beta)\in\mathrm{End}(V)$, defined by
$$
\<\beta(x,y),\xi\>=\<A_\xi(\beta)x, y\>,\;\; x,y\in V.
$$
The space $\mathrm{Sym}(V\times V,W)$ is a complete metric 
space with respect to the norm $\lVert\cdot\rVert$ defined by
$$
\|\beta\|^2=\sum_{i,j=1}^n\|\beta(e_i,e_j)\|^2
=\sum_{a=1}^m\|A_{\xi_a}(\beta)\|^2,
$$
where $\{e_i\}_{1\leq i\leq n}$ and $\{\xi_a\}_{1\leq a\leq m}$ are 
orthonormal bases of $V$ and $W$, respectively.

A function $\varphi\colon\mathrm{Sym}(V\times V,W)\to\R$ is called 
{\it homogeneous of degree} $d\in\R^+$ if 
$\varphi(t\beta)=t^d\phi(\beta)$ for every $t\geq 0$ and 
any $\beta\in\mathrm{Sym}(V\times V,W)$. The following lemma, 
that may be useful for other purposes,
is crucial for the proofs of our results. 

\begin{lemma}\label{Main lemma}
Let $n\geq2, m\geq1, 0\leq p < n/2$ be integers and $d\in\R^+$.
Suppose that there is a map that assigns to 
each pair $V,W$ of vector spaces of dimensions $n,m$ respectively, 
and equipped with positive definite inner products, a nonnegative continuous 
function $\varphi_{V,W}\colon\mathrm{Sym}(V\times V,W)\to\R$ 
that is homogeneous of degree $d$ and satisfies the following conditions:
\begin{enumerate}[topsep=1pt,itemsep=1pt,partopsep=1ex,parsep=0.5ex,leftmargin=*, label=(\roman*), align=left, labelsep=-0.5em]
\item For any isometries $\mathsf{i}\colon\tilde V\to V$ and $\mathsf{j}\colon\tilde W\to W$ we have
$$
\varphi_{V,W}(\beta)=\varphi_{\tilde V,\tilde W}(\mathsf{j}^{-1}\circ\beta\circ(\mathsf{i}\times\mathsf{i}))
\;\;\text{for all}\;\;\beta\in\mathrm{Sym}(V\times V,W).
$$
\item If $\varphi_{V,W}(\beta)=0$ for some $\beta\in\mathrm{Sym}(V\times V,W)$, 
then the endomorphism $A_u(\beta)$ has an eigenvalue of multiplicity 
at least $n-p$ for every unit vector $u\in W$, the vanishing of which 
implies that $A_u(\beta)=0$.
\end{enumerate}
Then there exists a positive constant 
$\delta(n,m)$, depending only on $n$ and $m$, such 
that the following inequality holds for any vector spaces $V,W$ and all $\beta\in\mathrm{Sym}(V\times V,W)$ 
$$
\varphi_{V,W}(\beta)\geq\delta(n,m)\left(\psi_p(\beta)\right)^{d/n},
$$
where the function $\psi_p\colon\mathrm{Sym}(V\times V,W)\to\R$ is given by 
$$
\psi_p(\beta)=\int_{\Lambda_p(\beta)}\vert\det A_u(\beta)\vert dS_u,
\;\;\Lambda_p(\beta)=\left\{u\in\mathbb{S}^{m-1}:
p<\mathrm{Index}\, A_u(\beta)<n-p\right\}, 
$$
and $dS_u$ stands for the volume 
element{\footnote{In case $m=1$, integration reduces to summation.}} 
of the unit $(m-1)$-sphere $\Sf^{m-1}$ in $W$. 
\end{lemma}

\vspace{1ex}

To simplify the proof of the above lemma, we first prove the following result.

\begin{sublemma}\label{sublemma}
Let $\{\beta_r\}_{r\in\N}$ be a sequence in 
$\psi_p^{-1}(1)\subset\mathrm{Sym}(V\times V,W)$ 
and $\{s_r\}_{r\in\N}$ a sequence 
of positive real numbers. Assume that the sequence 
$\gamma_r=s_r\beta_r$ 
converges to some $\gamma\in\mathrm{Sym}(V\times V,W)$
with $\varphi_{V,W}(\gamma)=0$. Then $\gamma=0$.
\end{sublemma}

\begin{proof}
For every $r\in\N$, since $\beta_r\in\psi_p^{-1}(1)$ and 
$\Lambda_p(\beta_r)=\Lambda_p(\gamma_r)$,
there exists an open subset $\CU_r\subset\Sf^{m-1}\subset W$ 
such that $\CU_r\subset\Lambda_p(\gamma_r)$ and 
$\det A_\xi(\gamma_r)\neq0$ for all $\xi\in\CU_r$.  
Let $\{u_r\}$ be any convergent sequence such that $u_r\in\CU_r$ 
for all $r\in\mathbb{N}$, and set $u=\lim_{r\rightarrow\infty} u_r$. 
We claim that 
\be\label{limbeta}
\lim_{r\rightarrow\infty} A_{u_r}(\gamma_r)=0.
\ee 
Since $\varphi_{V,W}(\gamma)=0$, by our assumption $A_u(\gamma)$ 
has an eigenvalue $\mu(u)$ with multiplicity at least $n-p$. 

We argue that $\mu(u)=0$. Using 
that $\lim_{r\rightarrow\infty}A_{u_r}(\gamma_r)=A_u(\gamma) $ 
and since $u_r\in\CU_r$,  we obtain 
$\mathrm{Index}\, A_u(\gamma)< n-p$ (see \cite{A}). 
Assuming that $\mu(u)< 0$, since its multiplicity is at least $n-p$, 
it follows that $\mathrm{Index}\, A_u(\gamma)\geq n-p$ and this 
is a contradiction. Therefore $\mu(u)\geq 0$ and this implies 
that $\mathrm{Index}\, A_u(\gamma)\leq p$.
Assuming now that $\mu(u)>0$, from 
$\lim_{r\rightarrow\infty}A_{u_r}(\gamma_r)=A_u(\gamma)$ it 
follows that $\mathrm{Index}\, A_{u_r}(\gamma_r)\leq p$
for $r$ large enough, and this contradicts the fact that 
$u_r\in\Lambda_p(\gamma_r)$. Therefore, $\mu(u)=0$ and 
the second condition for $\varphi_{V,W}$ yields that 
$\lim_{r\rightarrow\infty}A_{u_r}(\gamma_r)=A_u(\gamma)=0$.

We may choose convergent sequences 
$\{u_r^{(1)}\},\dots,\{u_r^{(m)}\}$ in $\CU_r$ 
such that the vectors $u_r^{(1)},\dots, u_r^{(m)}$ span 
$W$ for all $r\in\mathbb{N}$. Using the Gram-Schmidt process 
we obtain sequences $\{\xi_r^{(1)}\},\dots,\{\xi_r^{(m)}\}$ with
$\xi_r^{(a)}\in\span\{u_r^{(1)},\dots, u_r^{(a)}\}$ for each $1\leq a\leq m$. 
Then $\xi_r^{(a)}=\sum_{\ell=1}^a x_r^{(\ell)}u_r^{(\ell)}$, 
where $\{x_r^{(\ell)}\}$ are convergent sequences for each 
$1\leq\ell\leq m$. Consequently, we have that 
\be\label{beta1}
\gamma_r(\cdot,\cdot)
=\sum_{a=1}^m\<A_{{\xi}_r^{(a)}}(\gamma_r)\cdot,\cdot\>{\xi}_r^{(a)},\;\;\text{where}\;\; A_{{\xi}_r^{(a)}}(\gamma_r)= \sum_{\ell=1}^a x_r^{(\ell)} A_{u_r^{(\ell)}}(\gamma_r).
\ee
Using \eqref{limbeta}, we obtain 
$$
\lim_{r\rightarrow\infty}A_{{\xi}_r^{(a)}}(\gamma_r)
=\lim_{r\rightarrow\infty}\sum_{\ell=1}^a x_r^{(\ell)} A_{u_r^{(\ell)}}(\gamma_r)=0
\;\;\text{for all}\;\; 1\leq a\leq m.
$$
From \eqref{beta1} and the above it follows that 
$\lim_{r\rightarrow\infty}\gamma_r=0$ and thus, $\gamma=0$.\qed
\end{proof}

\vspace{1ex}

{\noindent{\it Proof of Lemma \ref{Main lemma}.}} 
Let $V,W$ be vector spaces as in the statement of the lemma. 
We first show that the set $\psi_p^{-1}(1)$ in nonempty. Indeed, 
since the function $\psi_p$ is homogeneous of degree $n$, 
it follows that for an arbitrary $\beta\in\mathrm{Sym}(V\times V, W)$ with 
$\psi_p(\beta)\neq 0$, we have that $\beta/(\psi_p(\beta))^{1/n}\in\psi_p^{-1}(1)$.

Let $\{\beta_r\}$ be a sequence in $\psi_p^{-1}(1)$ such that 
$$
\lim_{r\rightarrow\infty}\varphi_{V,W}(\beta_r)=\inf\varphi_{V,W}(\psi_p^{-1}(1))\geq 0.
$$

We claim that the sequence $\{\beta_r\}$ is bounded. 
Arguing indirectly, assume that there exists a subsequence, which by abuse of 
notation is again denoted by $\{\beta_r\}$, such that 
$\lim_{r\rightarrow\infty}\|\beta_r\|=\infty$. 
Since $\beta_r\in\psi_p^{-1}(1)$ implies that $\beta_r\neq0$ for all $r\in\N$, 
we can consider the sequence $\gamma_r=s_r\beta_r$, where $s_r=1/\|\beta_r\|$. 
Then $\{\gamma_r\}$ is bounded with $\|\gamma_r\|=1$ and, by passing to a 
subsequence if necessary, we may assume that it converges to some 
$\gamma\in\mathrm{Sym}(V\times V,W)$ with $\|\gamma\|=1$.
From $\varphi_{V,W}(\gamma_r)=\varphi_{V,W}(\beta_r)/\|\beta_r\|^d$ it follows that  
$\lim_{r\rightarrow\infty}\varphi_{V,W}(\gamma_r)=0$. 
Then, the continuity of $\varphi_{V,W}$ yields that $\varphi_{V,W}(\gamma)=0$ 
and Sublemma \ref{sublemma} implies that $\gamma=0$.
This contradicts the fact that $\|\gamma\|=1$, and the proof of the claim follows.

Thus we may assume that $\{\beta_r\}$ converges to some 
$\beta_\infty\in\mathrm{Sym}(V\times V,W)$. 
We argue that $\varphi_{V,W}(\beta_\infty)>0$. Suppose to the contrary 
that $\varphi_{V,W}(\beta_\infty)=0$. Then, Sublemma \ref{sublemma} applied 
to $\gamma_r=\beta_r$ yields that $\beta_\infty=0$. 
Therefore $A_u(\beta_\infty)=0$ for every $u\in\Sf^{m-1}$.
On the other hand, since $\beta_r\in\psi_p^{-1}(1)$, the mean value 
theorem implies that there exists $\xi_r\in\Sf^{m-1}$ 
such that 
$$
1=\psi_p(\beta_r)\leq\int_{\Sf^{m-1}}\vert\det A_u(\beta_r)\vert dS_u 
=\left|\det A_{\xi_r}(\beta_r)\right|\Vol(\Sf^{m-1})\;\;\text{for all}\;\;r\in\mathbb{N}.
$$
Since the sequence $\{\xi_r\}$ is bounded, we may assume that 
it converges to some $\xi\in\Sf^{m-1}$. Then, by letting $r\to\infty$ in 
the above, we conclude that $\det A_{\xi}(\beta_\infty)\neq0$
which contradicts the fact that $A_u(\beta_\infty)=0$ for every $u\in\Sf^{m-1}$.

Thus, the function $\varphi_{V,W}$ attains a positive minimum 
$\delta(n,m)=\varphi_{V,W}(\beta_\infty)$ on $\psi_p^{-1}(1)$, 
which by condition $(i)$ depends only on $n$ and $m$. 
Therefore, for an arbitrary $\beta\in\mathrm{Sym}(V\times V,W)$ 
with $\psi_p(\beta)\neq 0$, since 
$\gamma=\beta/(\psi_p(\beta))^{1/n}\in\psi_p^{-1}(1)$,
it follows that $\varphi_{V,W}(\gamma)\geq\delta(n,m)$.
Then, the homogeneity of $\varphi_{V,W}$ implies the desired inequality.
\qed

\section{Algebraic preliminaries}
Let $V$ and $W$ be real vector spaces of dimensions $n\geq 2$ 
and $m\geq 1$ respectively, equipped with positive definite inner products. 

The {\it Kulkarni-Nomizu product} of two bilinear forms 
$\phi,\psi\in\mathrm{Hom}(V\times V,\R)$ 
is the $(0,4)$-tensor 
$\phi\varowedge\psi\colon V\times V\times V\times V\to\R$ 
defined by 
\bea
\phi\varowedge\psi(x_1,x_2,x_3,x_4) &=& 
\phi(x_1,x_3)\psi(x_2,x_4)+\phi(x_2,x_4)\psi(x_1,x_3)\\
 & &-\phi(x_1,x_4)\psi(x_2,x_3)-\phi(x_2,x_3)\psi(x_1,x_4).
\eea
Using the inner product of $W$, we 
extend the Kulkarni-Nomizu product to bilinear forms 
$\beta,\gamma\in\Hom(V\times V,W)$, as the $(0,4)$-tensor 
$\beta\varowedge\gamma\colon V\times V\times V\times V\to\R$ 
defined by 
\bea
\beta\varowedge\gamma(x_1,x_2,x_3,x_4)\!\!\!&=&\!\!\! 
\langle\beta(x_1,x_3),\gamma(x_2,x_4)\rangle 
+\langle\beta(x_2,x_4),\gamma(x_1,x_3)\rangle\\
\!\!\!& &\!\!\! -\langle\beta(x_1,x_4),\gamma(x_2,x_3)\rangle-
\langle\beta(x_2,x_3),\gamma(x_1,x_4)\rangle.
\eea

For any $c\in\R$, we define the formal Ricci tensor as the map 
${\sf Ric}_c\colon\mathrm{Sym}(V\times V,W)\to\mathrm{Sym}(V\times V,\R)$ 
given by 
$$
{\sf Ric}_c(\beta)(x,y)=\tr\ {\sf R}_c(\beta)(\cdot,x,\cdot,y),\;\; x,y\in V,
$$
where 
$$
{\sf R}_c(\beta)=\frac{1}{2}\big(c\<\cdot,\cdot\>\varowedge\<\cdot,\cdot\>
+\beta\varowedge\beta\big).
$$

For any $\beta\in\mathrm{Sym}(V\times V,W)$, we denote by 
$\boldsymbol{\lambda}_{c,1}(\beta)\leq\dots\leq\boldsymbol{\lambda}_{c,n}(\beta)$ 
the eigenvalues of the self-adjoint operator  ${\sf T}_c(\beta)\in{\rm End}(V)$ 
determined by 
$$
\<{\sf T}_c(\beta) x,y\>=\frac{1}{n-1}{\sf Ric}_c(\beta)(x,y).
$$

For each $1\leq k\leq n$, we consider the functions 
$\boldsymbol{\rho}_{c,k}\colon\mathrm{Sym}(V\times V,W)\to\R$ 
defined by
$$
\boldsymbol{\rho}_{c,k}(\beta)=\frac{1}{k}\sum_{i=1}^k\boldsymbol{\lambda}_{c,i}(\beta).
$$
For $k=n$, we set $\boldsymbol{\rho}_c(\beta)=\boldsymbol{\rho}_{c,n}(\beta)$. 
Clearly we have 
$$
\boldsymbol{\rho}_c(\beta)=\frac{1}{n(n-1)}\tr\ {\sf Ric}_c (\beta).
$$ 

Furthermore, we define the function 
$\boldsymbol{\rho}^\perp\colon\mathrm{Sym}(V\times V,W)\to\R$ by 
$$
\boldsymbol{\rho}^\perp(\beta)=\frac{1}{n(n-1)}\|{\sf R}^\perp(\beta)\|=\frac{1}{n(n-1)}\Big(\sum_{i,j=1}^n\sum_{a,b=1}^m\left({\sf R}^\perp(\beta)(e_i,e_j,\xi_a,\xi_b)\right)^2\Big)^{1/2},
 $$ 
where $\{e_i\}_{1\leq i\leq n}$ and $\{\xi_a\}_{1\leq a\leq m}$ are orthonormal bases 
of $V$ and $W$, respectively. 
Here, the tensor 
 ${\sf R}^\perp(\beta)\colon V\times V\times W\times W\to\R$ is given by
 $$
 {\sf R}^\perp(\beta)(x,y,\xi,\eta)=\<[A_\xi(\beta),A_\eta(\beta)]x,y\>,\;\; x,y\in V,\;\xi,\eta\in W,
 $$
being $[A,B]=A\circ B-B\circ A$ the commutator of $A,B\in{\rm End}(V)$.
\vspace{1ex}

For every $\beta\in\mathrm{Sym}(V\times V,W)$, we set 
$\mathcal H_\beta=\left(1/n\right)\tr\beta, {\sf H}_\beta=\|\mathcal{H}_\beta\|$
and we denote by $\mathring{\beta}=\beta-\<\cdot,\cdot\>{\cal H}_\beta$ 
the traceless part of $\beta$. The form $\beta$ is called {\emph {umbilical}} 
if $\mathring{\beta}=0$. It is clear that $\beta$ is umbilical if and only if 
$A_\xi(\beta)=\<{\cal H}_\beta,\xi\>\operatorname{Id}_V$ for every $\xi\in W$.

\begin{lemma}\label{curvatures}
For every $\beta\in\mathrm{Sym}(V\times V,W)$ and any $c\in\R$,  
the following hold:
\begin{enumerate}[topsep=1pt,itemsep=1pt,partopsep=1ex,parsep=0.5ex,leftmargin=*, label=(\roman*), align=left, labelsep=-0.5em]
\item The formal Ricci tensor is given by
$$
{\sf Ric}_c(\beta)(x, y)=c(n-1)\<x,y\>+n\<{\cal H}_\beta,\beta(x,y)\>  
-\sum_{i=1}^n\<\beta(x,e_i),\beta(y,e_i)\>, 
$$
for $x,y\in V$, where $\{e_i\}_{1\leq i\leq n}$ is an orthonormal basis of $V$.

\item The DDVV inequality is valid
$$
\boldsymbol{\rho}_c(\beta)\leq c+{\sf H}_\beta^2-\boldsymbol{\rho}^\perp(\beta).
$$
Equality holds in the above inequality if and only if there exist orthonormal bases $\{e_i\}_{1\leq i\leq n}$ 
of $V$ and $\{\xi_a\}_{1\leq a\leq m}$ of $W$ such that 
$$
A_{\xi_1}(\beta)
=\operatorname{diag}(\lambda_1+\mu,\lambda_1-\mu,\lambda_1,\dots,\lambda_1),\;\; 
A_{\xi_2}(\beta)
=\operatorname{diag}\left(\left(\begin{array}{cc}\lambda_2 & \mu \\ \mu & \lambda_2 \end{array}\right),
\lambda_2, \dots,\lambda_2\right),
$$
and $A_{\xi_a}(\beta)=\lambda_a I_n$ for $3\leq a\leq m$, where 
$\mu,\lambda_1,\dots,\lambda_m$ are real numbers.
\end{enumerate}
\end{lemma}

\begin{proof} 
Part (i) follows by a straightforward computation. 
For the second part, we choose an orthonormal basis $\{u_a\}_{1\leq a\leq m}$ of $W$ such that $u_1={\cal H}_\beta/{\sf H}_\beta$ if ${\cal H}_\beta\neq0$.
Let $B_{u_a}(\beta)=A_{u_a}(\beta)-\<{\cal H}_\beta, u_a\>\operatorname{Id}_V$ be the traceless part of $A_{u_a}(\beta)$ for $1\leq a\leq m$. 
Using part (i) and taking into account that $\|\mathring{\beta}\|^2=\|\beta\|^2-n{\sf H}^2_\beta$, a direct computation yields that 
\begin{equation}\label{Bsc}
c+{\sf H}_\beta^2-\boldsymbol{\rho}_c(\beta)
= \frac{\|\mathring{\beta}\|^2}{n(n-1)}
=\frac{1}{n(n-1)}\sum_{a=1}^m\|B_{u_a}(\beta)\|^2.
\end{equation}
On the other hand, it follows easily that
\begin{equation}\label{Bscp}
\boldsymbol{\rho}^\perp(\beta)
= \frac{\|R^\perp(\beta)\|}{n(n-1)}
=\frac{1}{n(n-1)}\Big(\sum_{a,b=1}^m\|[B_{u_a}(\beta),B_{u_b}(\beta)]\|^2\Big)^{1/2}.
\end{equation}
Therefore, the desired inequality is equivalent to
$$
\sum_{a,b=1}^m\|[B_{u_a}(\beta),B_{u_b}(\beta)]\|^2\leq\Big(\sum_{a=1}^m\|B_{u_a}(\beta)\|^2\Big)^2.
$$
The last inequality, along with its equality case, has been proved in 
\cite{GT, Lu}. The remainder of the proof follows the same steps as the proof of 
Corollary 1.2 in \cite{GT}, with $T_pM$ replaced by $V$ 
and $T^\perp_pM$ replaced  by $W$.\qed
\end{proof}

\begin{proposition}\label{k-inequality}
The following assertions hold for any vector spaces $V,W$, any 
$\beta\in\mathrm{Sym}(V\times V,W)$ and every $c\in\R$:
\begin{enumerate}[topsep=2pt,itemsep=2pt,partopsep=2ex,parsep=0.5ex,leftmargin=*, label=(\roman*), align=left, labelsep=-0em]
\item If $n\geq 3$, then for every $k\in\{1,\dots,n-1\}$, the 
following inequality holds:
\begin{equation}\label{kineq}
\boldsymbol{\rho}_{c,k}(\beta)\leq c+{\sf H}_\beta^2-\boldsymbol{\rho}^\perp(\beta).
\end{equation}
Equality holds for some $k\in\{1,\dots,n-1\}$ if and only if 
$\beta$ is umbilical.
\item For every $\lambda\in [0,1)$, the following inequality holds:
\begin{equation}\label{lam}
\boldsymbol{\rho}_c(\beta)\leq c+{\sf H}_\beta^2-\lambda\boldsymbol{\rho}^\perp(\beta).
\end{equation}
Equality holds for some $\lambda\in [0,1)$ if and only if $\beta$ is umbilical.
\end{enumerate}
\end{proposition}

\begin{proof}
\emph{(i)} Assume that $n\geq 3$. For every $k\in\{1,\dots,n-1\}$ we have that
\begin{equation}\label{sinequ}
\boldsymbol{\rho}_{c,k}(\beta)\leq\boldsymbol{\rho}_{c,k+1}(\beta), 
\end{equation}
and equality holds if and only if 
$\boldsymbol{\lambda}_{c,1}(\beta)=\dots =\boldsymbol{\lambda}_{c,k+1}(\beta)$.
The proof of \eqref{kineq} then follows from the DDVV inequality 
in Lemma \ref{curvatures}$(ii)$.

Assume now that for some $k\in\{1,\dots,n-1\}$, inequality \eqref{kineq} 
holds as equality. From \eqref{sinequ} it follows that all inequalities in
$$
\boldsymbol{\rho}_{c,k}(\beta)\leq\dots\leq\boldsymbol{\rho}_{c,n}(\beta)\leq c
+ {\sf H}_\beta^2-\boldsymbol{\rho}^\perp(\beta)
$$
hold as equalities. This implies that all eigenvalues of ${\sf T}_c(\beta)$ 
are equal and thus, ${\sf Ric}_c(\beta)=(n-1)\rho\langle\cdot,\cdot\rangle$, 
where $\rho=\boldsymbol{\rho}_{c,i}(\beta)$, 
$1\leq i\leq n$.

Moreover, the DDVV inequality holds as equality for $\beta$.
Let $\{e_i\}_{1\leq i\leq n}$ and $\{\xi_a\}_{1\leq a\leq m}$ be orthonormal bases 
of $V$ and $W$, respectively, as in Lemma \ref{curvatures}$(ii)$. 
Using Lemma \ref{curvatures}
and taking into account that ${\cal H}_\beta=\sum_{a=1}^m\lambda_a\xi_a$, 
a straightforward computation yields that 
\begin{eqnarray*}
{\sf Ric}_c(\beta)(e_1, e_1)&=&(n-1)(c+ {\sf H}_\beta^2)+ \mu \left((n-2)\lambda_1-2\mu\right),\\
{\sf Ric}_c(\beta)(e_2, e_2)&=&(n-1)(c+ {\sf H}_\beta^2)- \mu \left((n-2)\lambda_1+2\mu\right),\\
{\sf Ric}_c(\beta)(e_i, e_i)&=& (n-1)(c+ {\sf H}_\beta^2), \;\; \text{if} \;\; i\geq 3.
\end{eqnarray*}
Since ${\sf Ric}_c(\beta)=(n-1)\rho\langle\cdot,\cdot\rangle$, the above 
yield that $\mu=0$. Therefore, we obtain that 
$A_{\xi_a}(\beta)=\lambda_a I_n,1\leq a\leq m$,
and this implies that $\beta$ is umbilical. The converse is obvious. 

\emph{(ii)} The desired inequality follows immediately from the 
DDVV inequality. If \eqref{lam} holds as equality, then
$$
c+ {\sf H}_\beta^2-\boldsymbol{\rho}_c(\beta)-\boldsymbol{\rho}^\perp(\beta)
=(\lambda-1)\boldsymbol{\rho}^\perp(\beta).
$$
By virtue of Lemma \ref{curvatures}$(ii)$, the left hand side of 
the above is nonnegative. Thus $\boldsymbol{\rho}^\perp(\beta)=0$ and 
the DDVV inequality holds as equality for $\beta$.
Now we choose the orthonormal bases of $V$ and $W$ as in Lemma 
\ref{curvatures}$(ii)$. Then a direct  computation yields that 
$\|{\sf R}^\perp (\beta)\|=4\mu^2$. Since $\boldsymbol{\rho}^\perp(\beta)=0$, 
it follows that $\mu=0$. Then, as in the proof of part $(i)$, we conclude that 
$\beta$ is umbilical.\qed
\end{proof}

\begin{proposition}\label{mainprop}
Given integers $n\geq 2$ and $m\geq1$, the following assertions hold: 
\begin{enumerate}[topsep=2pt,itemsep=2pt,partopsep=2ex,parsep=0.5ex,leftmargin=*, label=(\roman*), align=left, labelsep=-0em]
\item If $n\geq 3$, then for any $c\in\R$ and $k\in\{1,\dots,n-1\}$ there exists a constant 
$\delta_{c,k}(n,m)>0$ depending only on $n$ and $m$,
such that the following inequality holds
$$
c+{\sf H}_\beta^2-\boldsymbol{\rho}^\perp(\beta)-\boldsymbol{\rho}_{c,k}(\beta)
\geq\delta_{c,k}(n,m)\left(\psi_0(\beta)\right)^{2/n}
$$
for any vector spaces $V,W$ and all $\beta\in\mathrm{Sym}(V\times V,W)$, where $\psi_0$ is the 
function defined in Lemma \ref{Main lemma}.
\item For every $c\in\R$ and $\lambda\in [0,1)$ there exists a constant 
$\delta_{c,\lambda}(n,m)>0$ depending only on $n$ and $m$,
such that  the following inequality holds 
$$
 c+{\sf H}_\beta^2-\lambda\boldsymbol{\rho}^\perp(\beta) 
-\boldsymbol{\rho}_c(\beta)\geq\delta_{c,\lambda}(n,m)\left(\psi_0(\beta)\right)^{2/n}
$$
for any vector spaces $V,W$ and all $\beta\in\mathrm{Sym}(V\times V,W)$.
\end{enumerate}
\end{proposition}

\begin{proof}
For each pair of vector spaces $V,W$ of dimensions $n,m$, respectively, 
$c\in\R$, $\lambda\in [0,1]$ and $k\in\{1,\dots,n\}$, we consider 
the function $\varphi_{V,W}^{k,\lambda}\colon\mathrm{Sym}(V\times V,W)\to\R$ 
given by 
$\varphi_{V,W}^{k,\lambda}(\beta)=c+{\sf H}_\beta^2
-\lambda\boldsymbol{\rho}^\perp(\beta)-\boldsymbol{\rho}_{c,k}(\beta)$. 
We claim that the function $\varphi_{V,W}^{k, 1}$ for any $k\in\{1,\dots,n-1\}$, 
$n\geq3$, as well the function $\varphi_{V,W}^{n,\lambda}$ for any 
$\lambda\in [0,1)$ fulfill  the requirements 
in Lemma \ref{Main lemma}.

First, we argue that  $\varphi_{V,W}^{k,\lambda}$ is continuous. Let $\{\beta_r\}$ 
be a sequence in $\mathrm{Sym}(V\times V,W)$ with 
$\lim_{r\to\infty}\beta_r=\beta\in\mathrm{Sym}(V\times V,W)$ and we 
consider orthonormal bases $\{e_i\}_{1\leq i\leq n}$ and $\{\xi_a\}_{1\leq a\leq m}$
of $V$ and $W$, respectively. From Lemma \ref{curvatures}(i) it follows that the Ricci tensors satisfy
$\lim_{r\to\infty}{\sf Ric}_c(\beta_r)(e_i, e_j)={\sf Ric}_c(\beta)(e_i,e_j)$, $1\leq i,j\leq n$. 
Therefore $\lim_{r\to\infty}{\sf T}_c(\beta_r)={\sf T}_c(\beta)$ for 
the associated self-adjoint operators. Then the ordered eigenvalues of 
these operators satisfy
$\lim_{r\to\infty}\boldsymbol{\lambda}_{c,i}(\beta_r)=\boldsymbol{\lambda}_{c,i}(\beta)$, 
$1\leq i\leq n$ (see for instance \cite[Theorem 1]{A}).
Therefore,
$$
\lim_{r\to\infty}\boldsymbol{\rho}_{c,k}(\beta_r)
=\boldsymbol{\rho}_{c,k}(\beta)\;\;\text{for every}\;\; k\in\{1,\dots,n\}.
$$
Using the fact that $\lim_{r\to\infty}A_{\xi_a}(\beta_r)=A_{\xi_a}(\beta)$ 
for any $1\leq a\leq m$, we obtain 
$\lim_{r\to\infty}\boldsymbol{\rho}^\perp(\beta_r)=\boldsymbol{\rho}^\perp(\beta)$.
It is clear that $\lim_{r\to\infty}{\sf H}^2_{\beta_r}={\sf H}^2_\beta$. 
Hence the above imply that 
$\lim_{r\to\infty}\varphi_{V,W}^{k,\lambda}(\beta_r)=\varphi_{V,W}^{k,\lambda}(\beta)$ 
and thus the function $\varphi_{V,W}^{k,\lambda}$ is continuous.

We next show that $\varphi_{V,W}^{k,\lambda}$ is a homogeneous function of 
degree $d=2$. Lemma \ref{curvatures}$(i)$ yields that 
$$
{\sf Ric}_c(\beta)(x,y)
=c(n-1)\<x,y\>+{\sf Ric}_0(\beta)(x,y)\;\;\text{for every}\;\; x,y\in V.
$$
This implies that ${\sf T}_c{(\beta)}=c\operatorname{Id}_V+{\sf T}_0{(\beta)}$ 
and the corresponding eigenvalues are related by
$\boldsymbol{\lambda}_{c,i}(\beta)=c+\boldsymbol{\lambda}_{0,i}(\beta),1\leq i\leq n$. 
Therefore, $\boldsymbol{\rho}_{c,k}(\beta)=c+\boldsymbol{\rho}_{0,k}(\beta)$ 
and thus
$$
\varphi_{V,W}^{k,\lambda}= {\sf H}_\beta^2-\lambda\boldsymbol{\rho}^\perp(\beta)
-\boldsymbol{\rho}_{0,k}(\beta),\;\; 1\leq k\leq n.
$$
From Lemma \ref{curvatures}(i) it follows that 
${\sf Ric}_0(t\beta)=t^2{\sf Ric}_0(\beta)$ 
for every $t\in\R$. Then a straightforward computation shows that 
$\varphi_{V,W}^{k,\lambda}$ is a 
homogeneous function of degree $d=2$.

The fact that the function $\varphi_{V,W}^{k,\lambda}$ satisfies condition $(i)$ 
in Lemma \ref{Main lemma} follows directly from its definition.
Moreover, Proposition \ref{k-inequality} implies that
condition $(ii)$ is fulfilled for $p=0$. The proof now follows from Lemma \ref{Main lemma}.
\qed
\end{proof}

\vspace{1ex}

The following example shows that the inequality in part $(i)$ of Proposition 
\ref{mainprop} fails for $k=n$. In fact, this example shows that 
$\liminf_{\lambda\to 1^-}\delta_{c,\lambda}(n,m)=0$.

\begin{example}
\em{
Given $\mu>0$ and a sequence $\sigma_r>0$ such that 
$\lim_{r\to\infty}\sigma_r=0$,  we define the sequence 
$\beta_r\in\mathrm{Sym}(V\times V,W),r\in\mathbb{N}$, given by 
$\beta_r=A_{\xi_1}(\beta_r)\xi_1 + A_{\xi_2}(\beta_r)\xi_2$, where
$$
A_{\xi_1}(\beta_r)=\operatorname{diag}(\mu,-\mu,-\sigma_r,\sigma_r,\dots,\sigma_r),\;\; 
A_{\xi_2}(\beta_r)=\operatorname{diag}\left(\left(\begin{array}{cc} 0 & \mu \\ \mu & 0 \end{array}\right),  
-\sigma_r,\sigma_r,\dots,\sigma_r\right),
$$
with respect to orthonormal bases $\{e_i\}_{1\leq i\leq n}$ and 
$\{\xi_a\}_{1\leq a\leq m}$ of $V=\R^n$ and $W=\R^m$, respectively, 
with $n\geq 3$ and $m\geq 2$.

Then, for any unit vector $u=\sum_{a=1}^mu_a\xi_a\in W$ the 
eigenvalues of $A_{u}(\beta_r)$ are
\begin{eqnarray}
\tau_1&=&\mu(u_1^2+u_2^2)^{1/2},\;\;\tau_2=-\mu(u_1^2+u_2^2)^{1/2}, \nonumber \\
\tau_3&=&-\sigma_r(u_1+u_2),\;\;\tau_i=\sigma_r(u_1+u_2),\;\; 4\leq i\leq n. \nonumber
\end{eqnarray}
Thus $\mathrm{Index}\, A_u(\beta_r)=2$ if $u_1+u_2>0$, whereas 
$\mathrm{Index}\, A_u(\beta_r)=n-2$ if $u_1+u_2<0$.
Therefore, $\Lambda_1(\beta_r)=U$ for every $r\in\mathbb{N}$, 
where $U=\left\{u\in\mathbb{S}^{m-1}: u_1+u_2\neq 0\right\}$ and 
consequently, 
\begin{equation}\label{C}
\psi_1(\beta_r)=I\sigma_r^{n-2},\;\;\text{where}\;\; I
=\mu^2\int_U (u_1^2+u_2^2)|u_1+u_2|^{n-2}dS_u.
\end{equation}
On the other hand, a straightforward computation using \eqref{Bsc} and 
\eqref{Bscp}
yields that  
$$
c+{\sf H}_{\beta_r}^2-\boldsymbol{\rho}^\perp(\beta_r)
-\boldsymbol{\rho}_{c,n}(\beta_r)
=D(n)\sigma_r^2,\;\;\text{where}\;\; D(n)=\frac{4(3n-8)}{n^2(n-1)}.
$$
From \eqref{C} and the above it follows that the quotient
$$
\frac{c+{\sf H}_{\beta_r}^2-\boldsymbol{\rho}^\perp(\beta_r)
-\boldsymbol{\rho}_{c,n}(\beta_r)}{\left(\psi_1(\beta_r)\right)^{2/n}}
=\frac{D(n)}{I^{2/n}}\sigma_r^{4/n}
$$
tends to zero. Therefore, the function 
$(c+{\sf H}^2-\boldsymbol{\rho}^\perp-\boldsymbol{\rho}_{c,n})/\psi_1^{2/n}$
is not bounded from below by a positive constant. Since 
$\psi_0(\beta)\geq\psi_1(\beta)$ for every 
$\beta\in\mathrm{Sym}(V\times V,W)$, we conclude that this also 
holds for the function 
$(c+{\sf H}^2-\boldsymbol{\rho}^\perp-\boldsymbol{\rho}_{c,n})/\psi_0^{2/n}$.
}
\end{example}

\section{Proofs of the Results}

Let $f\colon M^n\to\R^{n+m}$ be an isometric immersion of a 
compact Riemannian manifold into the Euclidean space $\R^{n+m}$ 
equipped with the usual inner product $\langle\cdot,\cdot\rangle$. 
We denote by $N_fM$ the normal bundle of $f$ and by 
$\alpha_f\in\Gamma (\mathrm{Hom}(TM\times TM,N_fM))$ 
its second fundamental form.

First, let's recall some well-known facts about total curvature 
and how Morse theory imposes constraints on the Betti numbers.
The unit normal bundle of $f$ is defined as the set 
$$
UN_f=\left\{(p,\xi)\in N_fM:\|\xi\|=1\right\}.
$$
The {\it generalized Gauss map} $\nu\colon UN_f\to\Sf^{n+m-1}$ 
is given by $\nu(p,\xi)=\xi$. For each $u\in\Sf^{n+m-1}$, we 
consider the height function $h_u$ defined by 
$h_u(p)=\langle f(p),u\rangle,\, p\in M^n$. Since $h_u$ has a 
degenerate critical point if and only if $u$ is a critical value of 
$\nu$, by Sard's theorem, there exists a subset $E\subset\Sf^{n+m-1}$ 
of measure zero such that $h_u$ is a Morse function for all 
$u\in\Sf^{n+m-1}\smallsetminus E$. We denote by $\mu_i(u)$ 
the number of critical points of $h_u$ of index $i$ for each 
$u\in\Sf^{n+m-1}\smallsetminus E$ and set $\mu_{i}(u)=0$ 
for every $u\in E$. Following Kuiper \cite{Kuiper}, we define the 
{\it total curvature of index $i$ of $f$} by 
$$
\tau_i(f)=\frac{1}{\mathrm{Vol}(\Sf^{n+m-1})}\int_{\Sf^{n+m-1}}\mu_i(u) dS,
$$
where $dS$ denotes the volume element of the sphere $\Sf^{n+m-1}$.
From the weak Morse inequalities \cite[Theorem 5.2, p. 29]{Milnor63}, 
we have 
\begin{equation}\label{wmi}
\mu_i(u)\geq b_i(M^n;\Fi)\;\;\text{for all}\;\; u\in\Sf^{n+m-1}\smallsetminus E,
\end{equation}
where $b_i(M^n;\Fi)$ is the $i$-th Betti number of $M^n$ 
over an arbitrary coefficient field $\Fi$. By integrating over $\Sf^{n+m-1}$, 
we obtain 
\begin{equation}\label{TC}
\tau_i(f)\geq b_i(M^n;\Fi).
\end{equation}

There is a natural volume element $d\varSigma$ on the unit normal 
bundle $UN_f$. In fact, if $dV$ is a $(m-1)$-form on $UN_f$ such 
that its restriction to a fiber of the unit normal bundle at $(p,\xi)$ 
is the volume element of the unit $(m-1)$-sphere of the normal 
space of $f$ at $p$, then $d\varSigma=dM\wedge dV$, where 
$dM$ is the volume element of $M^n$. 
Shiohama and Xu \cite[p. 381]{SX} refined a well-known integral formula due to Chern-Lashof \cite{CL1, CL2}, and proved that 
\begin{equation}\label{ShXu}
\int_{U^iN_f}\left|\mathrm{det}A_\xi\right| d\varSigma
=\int_{\Sf^{n+m-1}}\mu_i(u) dS,
\end{equation}
where $U^iN_f$ is the subset of the unit normal bundle of $f$ 
defined by
$$
U^iN_f=\left\{(p,\xi)\in UN_f:\mathrm{Index}\, A_\xi=i\right\},\;\; 0\leq i\leq n,
$$
and $A_\xi$ the shape operator of $f$ with respect to $\xi$, 
where $(p,\xi)\in UN_f$.

\begin{lemma}\label{sumti}
Let $f\colon M^n\rightarrow\R^{n+m},n\geq 2$, be an 
isometric immersion of a compact Riemannian manifold 
such that $\sum_{i=1}^{n-1}\tau_i(f)<2$. 
Then, for any coefficient field $\Fi$, 
the Betti numbers satisfy $\sum_{i=1}^{n-1} b_i(M^n;\Fi)\in \{0,1\}$. In particular: 
\begin{enumerate}[topsep=3pt,itemsep=2pt,partopsep=2ex,parsep=0.5ex,leftmargin=*, label=(\roman*), align=left, labelsep=-0em]
\item If $b_i(M^n;\Fi)=0$ for any $1\leq i\leq n-1$, then $M^n$ is homeomorphic to $\Sf^n$.
\item If $\sum_{i=1}^{n-1} b_i(M^n;\Fi)=1$, then $M^n$ is an Eells-Kuiper manifold.
\end{enumerate}
\end{lemma}

\begin{proof}
The assumption is equivalent to
$$
\int_{\Sf^{n+m-1}}\sum_{i=1}^{n-1}\mu_i(u)dS<2\mathrm{Vol}(\Sf^{n+m-1}).
$$
This implies that there exists a unit vector $u_0$ such that the 
height function $h_{u_0}$ is a Morse function satisfying
$\sum_{i=1}^{n-1}\mu_i(u_0)<2$, or equivalently
\begin{equation}\label{smi}
\sum_{i=1}^{n-1}\mu_i(u_0)\leq1.
\end{equation}
The above and \eqref{wmi} yield that $\sum_{i=1}^{n-1} b_i(M^n;\Fi)\in\{0,1\}$  
for every coefficient field $\Fi$.

{\emph{(i)}} Suppose that 
\begin{equation}\label{sb0}
b_i(M^n;\Fi) =0\;\;{\text {for every}}\;\; 1\leq i\leq n-1.
\end{equation}

We claim that the homology groups of $M^n$ over the integers satisfy 
$H_k(M^n;\Z)=0$ for any $1\leq k\leq n-1$, which in view of \eqref{sb0},  
is equivalent to the claim that $M^n$  has no torsion. Indeed, 
if $H_k(M^n;\Z)$ contains torsion for some $1\leq k\leq n-1$, then 
$H_k(M^n;\Z_p)\neq0$ for some prime $p$, which contradicts \eqref{sb0}. 
Hence $M^n$ is a homology sphere over the integers. 

We now prove that $M^n$ is simply connected. This is clear if $n=2$. 
Assume that $n\geq3$ and suppose, to the contrary,  
that the fundamental group is $\pi_1(M^n)\neq 0$.
It follows from \cite[Proposition 4.5.7, p. 90]{AD} that $\mu_1(u_0)\neq 0$.  
Thus, \eqref{smi} yields that $\mu_1(u_0)=1$ and $\mu_i(u_0)=0$ 
for any $2\leq i\leq n-1$. Then, by Morse theory, it follows that the 
manifold $M^n$ has the homotopy type of a CW-complex with no cells 
of dimension $2\leq i\leq n-1$. In particular, there are no $2$-cells, 
and thus, by the cellular approximation theorem, the inclusion 
of the $1$-skeleton $\mathrm{X}^{(1)}\hookrightarrow M^n$ 
induces isomorphism between the fundamental groups. Therefore, 
$\pi_1(M^n)$ is a free group on $b_1(M^n;\Z)=0$ elements. Hence, 
$\pi_1(M^n)=0$, which is a contradiction.

Thus, $M^n$ is a simply connected homology sphere over the 
integers and therefore a homotopy sphere. By the (generalized) 
Poincar\'e conjecture (Smale $n\geq 5$, Freedman $n=4$, 
Perelman $n=3$), $M^n$ is 
homeomorphic to $\Sf^n$.

{\emph{(ii)}} Assume that $\sum_{i=1}^{n-1} b_i(M^n;\Fi) =1.$
From Poincar\'e duality it follows that $n$ is even and the Betti 
numbers are
$$
b_{n/2}(M^n;\Fi) =1\;\;\mbox{and}\;\; b_i(M^n;\Fi)
=0,\; 1\leq i\leq n-1,\; i\neq n/2.
$$
Jointly with \eqref{wmi} and \eqref{smi}, the above implies that
\begin{equation} \label{mis}
\mu_{n/2}(u_0) =1\;\;\mbox{and}\;\;\mu_i(u_0)
=0,\; 1\leq i\leq n-1,\; i\neq n/2.
\end{equation}
Since $\mu_1(u_0)=0$, it follows from \cite[Lemma 4.11, p. 85]{CE75} 
that $\mu_0(u_0)\leq 1$. Hence from \eqref{wmi}, we obtain 
\begin{equation}\label{m0}
\mu_0(u_0)=1.
\end{equation}
The Euler-Poincar\'e characteristic of 
$M^n$ is given by 
$$\chi(M^n)=\sum_{i=0}^n (-1)^i b_i(M^n;\Fi)=3.$$
On the other hand, using \eqref{mis} and \eqref{m0}, from \cite[Theorem 5.2, p. 29]{Milnor63} we have 
$$
\chi(M^n)=\sum_{i=0}^n (-1)^i\mu_i(u_0)=2+\mu_n(u_0), 
$$
and thus $\mu_n(u_0)=1$. Taking into account \eqref{mis} and 
\eqref{m0}, this implies that the height function $h_{u_0}$ is a Morse function with 
three critical points. Therefore, $M^n$ is an Eells-Kuiper manifold.
\qed
\end{proof}

\vspace{2ex}

{\noindent{\it Proof of Theorem \ref{main1}.}} 
Let $f\colon M^n\to\Q^{n+m}_c,n\geq 3$, be an isometric immersion 
with second fundamental form $\alpha_f$, mean curvature $H_f$ and 
shape operator $A_\xi$ with respect to $\xi$, where $(p,\xi)\in UN_f$. 
Assume first that the ambient space is the Euclidean space $\R^{n+m}$, 
that is $c=0$. Since 
$H_f(p)={\sf H}_{\alpha_f(p)}$, $\rho_k(p)=\boldsymbol{\rho}_{0,k}(\alpha_f(p))$ 
and $\rho^\perp(p)=\boldsymbol{\rho}^\perp(\alpha_f(p))$, it follows from 
Proposition \ref{mainprop}$(i)$ that 
$$
(H_f^2-\rho^\perp-\rho_k)^{n/2}(p)\geq 
\left(\delta_{0,k}(n,m)\right)^{n/2}\int_{\Lambda_0(\alpha_f(p))}\vert\det A_\xi\vert\ dV_\xi
$$
for all $p\in M^n$. Integrating over $M^n$ and using \eqref{ShXu}, we have 
\begin{equation*}
\int_{M^n}(H_f^2-\rho^\perp-\rho_k)^{n/2}dM\geq 
\left(\delta_{0,k}(n,m)\right)^{n/2}\Vol(\Sf^{n+m-1})\sum_{i=1}^{n-1}\tau_i(f).
\end{equation*}
Thus, from the above and (\ref{TC}) we obtain 
\begin{equation}\label{57667}
\int_{M^n}(H_f^2-\rho^\perp-\rho_k)^{n/2}dM
\geq\varepsilon_k(n,m)\sum_{i=1}^{n-1}\tau_i(f)
\geq\varepsilon_k(n,m)\sum_{i=1}^{n-1}b_i(M;\Fi),
\end{equation}
where $\varepsilon_k(n,m)=\left(\delta_{0,k}(n,m)\right)^{n/2}\Vol(\Sf^{n+m-1})$. 

Now, assume that
$$
\int_{M^n}(H_f^2-\rho^\perp -\rho_k)^{n/2}dM<2\varepsilon_k(n,m).
$$ 
Then it follows directly from $(\ref{57667})$ that $\sum_{i=1}^{n-1}\tau_i(f)<2$, 
and the proof follows from Lemma \ref{sumti}. 
In particular, if 
$$\int_{M^n}(H_f^2-\rho^\perp -\rho_k)^{n/2}dM<\varepsilon_k(n,m),$$ 
then $\sum_{i=1}^{n-1}\tau_i(f)<1$ and \eqref{TC} implies that only the first case in Lemma \ref{sumti} can occur.

Suppose now that $c>0$. We consider the isometric immersion 
$\tilde f\colon M^n\to\R^{n+m+1}$ given by $\tilde f= i\circ f$, where $i$ is an umbilical 
inclusion of the sphere $\Sf_c^{n+m}$ of radius $R=1/\sqrt{c}$ into the Euclidean space $\R^{n+m+1}$. Clearly 
$H^2_{\tilde f}=H_f^2+c$ and the normal curvature $\tilde\rho^\perp$ of $\tilde f$ is given by
$\tilde\rho^\perp=\rho^\perp$. Then the proof follows from the above argument applied to $\tilde f$ with 
$\varepsilon_k(n,m)=\left(\delta_{0,k}(n,m+1)\right)^{n/2}\Vol(\Sf^{n+m})$.
\qed

\begin{remark}\label{constants}
{\em
It is immediate from \eqref{sinequ} that the constants in Proposition 
\ref{mainprop}$(i)$ satisfy
$\delta_{c,1}(n,m)\geq\dots\geq\delta_{c,n-1}(n,m)>0$.
Then, it follows from the proof of Theorem 
\ref{main1} that the constants $\varepsilon_k(n,m),1\leq k\leq n-1$, 
satisfy
$\varepsilon_1(n,m)\geq\dots\geq\varepsilon_{n-1}(n,m)>0$.
}
\end{remark}

{\noindent{\it Proof of Corollary \ref{Wei}.}}
It follows immediately from Theorem \ref{main1} for $k=1$.\qed

\vspace{1ex}

{\noindent{\it Proof of Corollary \ref{Einstein}.}}
Since $M^n$ is Einstein and $n\geq 3$, it has constant normalized Ricci 
curvature equal to $\rho$. Therefore, $\rho_k=\rho$ for every 
$k\in\{1,\dots,n\}$. The proof follows immediately from Theorem 
\ref{main1}.\qed
\vspace{1ex}

{\noindent{\it Proof of Theorem \ref{main2}.}}
Let $f\colon M^n\to\Q^{n+m}_c,n\geq 2$, be an isometric 
immersion. We claim that the integral 
$$
\int_{M^n}(c+H_f^2-\lambda\rho^\perp-\rho)^{n/2}dM
$$
is invariant under conformal changes of the metric $\langle\cdot,\cdot\rangle$ 
of $\Q^{n+m}_c$. Indeed, 
from \eqref{Bsc} and \eqref{Bscp} it follows that
\begin{equation}\label{Bl}
n(n-1)(c+H_f^2-\lambda\rho^\perp-\rho)=\|\Phi_f\|^2-\lambda \|R^\perp\|,
\end{equation}
where $\Phi_f$ is the traceless part of the second fundamental form of $f$ and $R^\perp$ its normal curvature tensor.
Consider the conformal change 
$\widetilde{\langle\cdot,\cdot\rangle}=e^{2u}\langle\cdot,\cdot\rangle$ 
of the metric of $\Q^{n+m}_c$, where $u$ is a smooth function, 
and let $\tilde{f}\colon\tilde{M}^n\to (\Q^{n+m}_c,\widetilde{\langle\cdot,\cdot\rangle})$ 
be the isometric immersion induced by $f$. Then, at corresponding points the 
normal spaces of $f,\tilde{f}$ coincide and particularly, the second fundamental 
forms and the mean curvature vector fields 
are related by
$$
\alpha_{\tilde{f}}=\alpha_f-\langle\cdot,\cdot\rangle (\grad u)^\perp\;\; 
\text{and}\;\;{\cal H}_{\tilde f}=e^{-2u}\left({\cal H}_f-(\grad u)^\perp\right),
$$
where $(\grad u)^\perp$ is the normal component of the gradient of $u$ is 
with respect to the metric $\langle\cdot,\cdot\rangle$.
The above imply that $\Phi_{\tilde f}=\Phi_f$ and thus,
the traceless parts of the shape operators of $\tilde{f}, f$ associated to 
a local orthonormal frame field $\{\xi_a\}_{1\leq a\leq m}$ of $N_fM$ satisfy
$\tilde{B}_{\xi_a}=B_{\xi_a},1\leq a\leq m$.
Any orthonormal tangent frame field $\{e_i\}_{1\leq i\leq n}$ of 
$(M,\langle\cdot,\cdot\rangle)$ and any orthonormal normal frame field 
$\{\xi_a\}_{1\leq a\leq m}$ of $N_fM$ give rise to orthonormal frame fields 
$\{\tilde{e}_i\}_{1\leq i\leq n}$ of $(M,\widetilde{\langle\cdot,\cdot\rangle})$ 
with $\tilde{e}_i=e^{-u}e_i$ and $\{\tilde{\xi}_a\}_{1\leq a\leq m}$ of $N_{\tilde f}M$ with $\tilde{\xi}_a= e^{-u}\xi_a$. 
Therefore, a direct computation using \eqref{Bsc} and \eqref{Bscp} yields that 
$$
\|\Phi_{\tilde f}\|_{\sim}^2-\lambda \|\tilde{R}^\perp\|_{\sim} = e^{-2u}(\|\Phi_f\|^2-\lambda \|R^\perp\|).
$$
Now, by virtue of \eqref{Bl}, the claim follows from the above equality and the fact that the 
volume element of  $(M,\widetilde{\langle\cdot,\cdot\rangle})$ is $e^{nu}dM$. 

Therefore, under a conformal change of the metric of (a part of) $\Q^{n+m}_c$, 
we may assume without loss of generality that the ambient space is the 
Euclidean space $\R^{n+m}$. Since $H_f(p)={\sf H}_{\alpha_f(p)}$, $\rho(p)
=\boldsymbol{\rho}_{0}(\alpha_f(p))$ and $\rho^\perp(p)
=\boldsymbol{\rho}^\perp(\alpha_f(p))$, it follows from 
Proposition \ref{mainprop}$(ii)$ that 
$$
(H_f^2-\lambda\rho^\perp-\rho)^{n/2}(p)\geq 
\left(\delta_{0,\lambda}(n,m)\right)^{n/2}\int_{\Lambda_0(\alpha_f(p))}\vert\det A_\xi\vert\ dV_\xi
$$
for all $p\in M^n$. Then the proof follows arguing as in the proof of 
Theorem \ref{main1} with 
$\varepsilon_\lambda(n,m)=\left(\delta_{0,\lambda}(n,m)\right)^{n/2}\Vol(\Sf^{n+m-1})$.
\qed

\begin{remark}
{\em 
Since $\lambda\in [0,1)$, from \eqref{Bl} it follows that
$$
n(n-1)(c+H_f^2-\lambda\rho^\perp-\rho) \leq\|\Phi_f\|^2=\|\alpha_f\|^2-nH_f^2.
$$
This shows that our integrand is smaller than the (normalized) one 
of Shiohama and Xu in \cite{SX2}, which corresponds to $\lambda=0$.
}
\end{remark}

\section{Counterexamples to Theorem \ref{main1} for $k=n$}

\subsection{Minimal surfaces in spheres}

In this subsection, we discuss some properties of the isotropic surfaces 
in spheres, which are the basic tool for the construction of new 
compact 3-dimensional Wintgen ideal submanifolds.

Let $f\colon L^2\to\Sf^{n+2}$ denote an isometric immersion of a
two-dimensional oriented Riemannian manifold into the sphere $\Sf^{n+2}$.
The $k^{th}$\emph{-normal space} of $f$ at $x\in L^2$ for $k\geq 1$ is 
given by
$$
N^f_k(x)=\spa\big\{\a_f^{k+1}(X_1,\ldots,X_{k+1}):X_1,\ldots,X_{k+1}\in T_xL\big\}
$$
where $\a_f^2=\a_f\colon TL\times TL\to N_fL$ is the standard 
second fundamental form with values in the normal bundle and  
$$
\a_f^s\colon TL\times\cdots\times TL\to N_fL,\;\; s\geq 3, 
$$
is the symmetric tensor called
the $s^{th}$\emph{-fundamental form} defined inductively by
$$
\a_f^s(X_1,\ldots,X_s)=\big(\nabla^\perp_{X_s}\ldots
\nabla^\perp_{X_3}\a_f(X_2,X_1)\big)^\perp.
$$
Here  $\nabla^{\perp}$ is the induced connection in the normal bundle $N_fL$
and $(\;\;)^\perp$ means taking the projection onto the normal complement of 
$N^f_1\oplus\cdots\oplus N^f_{s-2}$ in $N_fL$.

Now suppose that $f\colon L^2\to\Sf^{n+2}$ denotes an oriented minimal and 
substantial surface. The latter means that the codimension cannot 
be reduced, in fact, not even locally since minimal surfaces are real 
analytic. It is known (cf. \cite{v,df}) that  the normal bundle 
$N_fL$ of $f$ splits along an open dense subset of $L^2$ as
$$
N_fL=N_1^f\oplus\dots\oplus N_m^f,\;\;\; m=[(n+1)/2],
$$
since all higher normal bundles  have rank two except possibly the last 
one that has rank one if $n$ is odd. The orientation of $L^2$ induces 
an orientation on each plane vector bundle $N_s^f$ given by the 
ordered pair 
$$
\xi_1^s=\a_f^{s+1}(X,\ldots,X),\;\;\;\xi_2^s=\a_f^{s+1}(JX,\ldots,X)
$$
where $0\neq X\in TL$ and $J$ is the complex structure of $L^2$ 
determined by the metric and orientation. 

For each $1\leq k\leq m$, the \emph{$k^{th}$-order curvature ellipse}
$\E^f_k(x)\subset N^f_{k}(x)$ is defined  by
$$
\E^f_k(x) = \big\{\alpha_f^{k+1}(Z^{\varphi},\ldots,Z^{\varphi})\colon\,
Z^{\varphi}=\cos\varphi Z+\sin\varphi JZ\;\mbox{and}\;\varphi\in[0,2\pi)\big\}
$$
where $Z\in T_xL$ is any vector of unit length.

A surface $f\colon L^2\to\Sf^{n+2}$ is called 
$r$-\emph{isotropic} if it is minimal and at any $x\in L^2$ 
and for any $1\leq k\leq r$ the ellipses of curvature $\E^f_k(x)$ 
contained in all two-dimensional  $N^f_k$$\,{}^{\prime}$s 
are circles.  
We point out that there are alternative ways to define 
isotropy for surfaces, for instance, in terms of the vanishing of higher order Hopf 
differentials \cite{v1}.

We consider the open and dense subset $L_1$ of $L^2$ defined by 
$$
L_1=\big\{x\in L^2:\dim N^f_1(x)\text{ is maximal}\big\}.
$$ Then 
the 1st normal spaces along $L_1$ form a subbundle $N^f_1|_{L_1}$ 
of the normal bundle $N_fL$. Inductively, we define the open and dense 
subset 
\begin{equation*}
L_s=\big\{x\in L_{s-1}:\dim N^f_s(x)\text{ is maximal}\big\}
\end{equation*}
and similarly the $s$-th normal spaces along $L_s$ form a subbundle 
$N^f_s|_{L_s}$ of the normal bundle. The following result was proved 
in \cite{Ch} (see also \cite{v,dv}).

\begin{proposition}\label{extend} 
If the surface $f$ is $r$-isotropic, then each 
$L^2\smallsetminus L_s, 1\leq s\leq r$, consists of isolated points 
and the vector bundle $N^f_s|_{L_s}$ smoothly extends 
to a vector bundle over $L^2$ still denoted by $N^f_s$.
\end{proposition}

\subsection{A class of compact 3-dimensional Wintgen submanifolds}

In the sequel, let $f\colon L^2\to\Sf^{2n+2},n\geq 2$, be a substantial  
\emph{$(n-1)$-isotropic} surface. Then Proposition \ref{extend} implies 
that each plane bundle $N^f_s|_{L_s}, 1\leq s \leq n-1$, smoothly extends 
to a plane bundle over $L^2$. Clearly, $N^f_n$ is the orthogonal complement of 
$N^f_1\oplus\dots\oplus N^f_{n-1}$ in $N_fL$. 
There are plenty of compact $(n-1)$-isotropic surfaces in $\Sf^{2n+2},n\geq 2$. 
This is the case of pseudoholomorphic curves in the nearly 
K\"ahler $\Sf^6$. In fact, it is known that they are $1$-isotropic \cite{Br,EschVl,v1}. 
Moreover, there are compact substantial pseudoholomorphic curves in the nearly 
K\"ahler $\Sf^6$ of positive genus (see \cite{Br}). Additionally, all $1$-isotropic tori 
in $\Sf^6$ were described in \cite{bpw}. 

The following result provides a method to produce new compact Wintgen ideal 
submanifolds. 

\begin{theorem}\label{Wint}
Let $f\colon L^2\to\Sf^{2n+2},n\geq 2$, be a compact oriented substantial 
$(n-1)$-isotropic surface of genus $g(L)$ and let $M^3$ be 
the total space of the unit bundle $p\colon UN^f_n\rightarrow L^2$ of 
the plane bundle $N^f_n$. 
Then $\phi_f\colon M^3\to\Sf^{2n+2}$ given by 
$\phi_f(x,w)=w$ is a minimal Wintgen ideal submanifold, 
whose first Betti number satisfies $b_1(M^3)\geq2 g(L)$.
\end{theorem}
\proof
It follows from Proposition 4.2 in \cite{dksv} 
that the map  $\phi_f\colon M^3\to\Sf^{2n+2}$ is a minimal immersion 
of rank two with polar surface $f$. Moreover, Proposition 8 in \cite{df} implies 
that the first curvature ellipse (as defined in \cite{df}) of $\phi_f$ is a circle at any point. It follows 
from Corollary 1.2 in \cite{GT} that a minimal $n$-dimensional submanifold 
of a space form with relative nullity $n-2$ is a Wintgen ideal submanifold if 
and only if its first curvature ellipse is a circle at any point. Hence $\phi_f$ 
is a Wintgen ideal submanifold. 

Now we argue that $b_1(M^3)\geq2 g(L)$. 
Clearly, the circle bundle $p\colon UN^f_n\rightarrow L^2$
is oriented (as defined in \cite[p.\ 114]{BT}). 
Then we have from \cite[Theorem 13.2, p. 390]{B} 
or \cite[p.\ 438]{Hat} that the cohomology rings of 
$M^3$ and $L^2$ are related by the following long exact 
sequence known as the Gysin sequence:
\be\label{gysin}
\!\!\!\cdots\xrightarrow{\Psi_{i-1}}
H^i(L;\Z)\xrightarrow{p^*_i}H^i(M;\Z)\xrightarrow{\Phi_i}
H^{i-1}(L;\Z)\xrightarrow{\Psi_i}H^{i+1}(L;\Z)\nonumber
\xrightarrow{p^*_{i+1}}\cdots
\ee
The exactness gives that $p^*_1$ is a monomorphism 
and ${\rm {Im}}\,p^*_1=\ker \Phi_1$. On the other hand we have 
$$
{\rm {Im}}\,p^*_1\cong H^1(L;\Z)/\ker\,p^* _1\cong H^1(L;\Z)
\cong\Z^{2g(L)}.
$$
Thus $\ker \Phi_1\cong\Z^{2g(L)}$. The desired inequality follows 
from the fact that $H^1(M;\Z)$ contains a subgroup isomorphic to 
$\Z^{2g(L)}$.\qed

\medskip

Theorem \ref{Wint} shows that each compact oriented substantial 
$(n-1)$-isotropic surface $f\colon L^2\to\Sf^{2n+2},n\geq 2$, of 
positive genus $g(L)$ gives rise to a compact 3-dimensional Wintgen ideal
submanifold with positive first Betti number. For instance, $f$ can be chosen 
to be a compact substantial pseudoholomorphic curve in the nearly 
K\"ahler $\Sf^6$ of positive genus (see \cite{Br}), or a $1$-isotropic torus  
in $\Sf^6$ as described in \cite{bpw}. Clearly, these Wintgen ideal submanifolds 
violate the inequality in Theorem \ref{main1} for $k=n=3$.

\vspace{10mm}

\noindent Christos-Raent Onti\\
Department of Mathematics and Statistics\\
University of Cyprus\\
1678, Nicosia -- Cyprus\\
e-mail: onti.christos-raent@ucy.ac.cy
\medskip

\noindent Kleanthis Polymerakis\\
University of Ioannina \\
Department of Mathematics\\
45110 Ioannina -- Greece\\
e-mail: kpolymerakis@uoi.gr
\medskip

\noindent Theodoros Vlachos\\
University of Ioannina \\
Department of Mathematics\\
45110 Ioannina -- Greece\\
e-mail: tvlachos@uoi.gr

\end{document}